\documentclass{amsart}

\usepackage[main=english]{babel} 
\usepackage{amsmath,amssymb,amscd,amsfonts,amsthm,amsrefs} 
\usepackage{microtype} 
\usepackage[babel,autostyle]{csquotes}

\newtheorem{theorem}{Theorem}[section]
\newtheorem{lemma}[theorem]{Lemma}
\newtheorem{proposition}[theorem]{Proposition}
\newtheorem{corollary}[theorem]{Corollary}

\theoremstyle{definition}

\theoremstyle{remark}

\numberwithin{equation}{section}

\newcommand{\iso}{\cong}
\newcommand{\semi}{\rtimes}
\newcommand{\normal}{\trianglelefteq}
\providecommand{\abs}[1]{\ensuremath{\lvert#1\rvert}}
\renewcommand{\phi}{\varphi}

\DeclareMathOperator{\inv}{inv} 
\DeclareMathOperator{\res}{res}
\DeclareMathOperator{\inc}{inc}
\DeclareMathOperator{\Syl}{Syl}

\begin{document}

\title{Local conjugacy in prosolvable groups}

\author{Michael C. Burkhart}
\curraddr{University of Chicago, Chicago, Illinois}
\email{burkh4rt@uchicago.edu}



\maketitle

\begin{abstract}
	For prosolvable groups, we provide conditions under which two locally conjugate supplements of a normal pronilpotent subgroup are conjugate, analogous to those given by Losey and Stonehewer for finite groups. Our proof relies on a primary-type decomposition of the first cohomology set which may be of independent interest. We also give local conditions for subgroup inclusion and a fixed-point result in the style of Glauberman.
\end{abstract}

\section{Introduction}

Two closed subgroups $H$ and $H'$ of a profinite group $G$ are \emph{locally conjugate} if for each prime $p$, a Sylow $p$-subgroup of $H$ is conjugate to a Sylow $p$-subgroup of $H'$. For a finite solvable group $G$, Losey and Stonehewer~\cite{Los79} proved two locally conjugate supplements of a normal nilpotent subgroup $N$ are conjugate if one of the following holds: (A) $G/N$ is nilpotent, (B) $N$ is abelian, or (C) the Sylow $p$-subgroups of $G$ have nilpotency class at most two. Evans and Shin~\cite{Eva88} showed that when $N$ is abelian, $G$ need not be solvable. Shin~\cite{Shi95} subsequently demonstrated that these results continue to hold for profinite groups $G$ and nilpotent $N$. In this work, we extend to pronilpotent $N$ as follows:

\begin{theorem}
	\label{thm:loc_conj}
	Let $N$ be a pronilpotent normal closed subgroup of a profinite group $G$. If either $G$ is prosupersolvable or $G/N$ is pronilpotent, then two closed supplements of $N$ are conjugate if and only if they are locally conjugate.
\end{theorem}

Our proof depends in part on a primary-type decomposition of the first cohomology. For a profinite group $J$, a discrete topological group $N$ is a $J$-group if $J$ operates continuously on it. In this context, cocycles are continuous maps $\phi: J \to N$ satisfying
$\phi(jj') = \phi(j) \phi(j')^{j^{-1}}$ for all $j,j' \in J$. Two cocycles $\phi$ and $\phi'$ are cohomologous if there exists $n \in N$ such that
$\phi'(j) = n^{-1} \phi(j) n^{j^{-1}}$ for all $j\in J$; we then write $\phi\sim\phi'$. The first cohomology $H^1(J,N)$ is the
set $Z^1(J,N)$ of cocycles modulo this equivalence relation with a distinguished point corresponding to the class containing the map taking each
element of $J$ to the identity of $N$.

When $N$ is abelian, $H^1(J,N)$ is torsion abelian and so admits a primary decomposition~\cite[Cor. 6.7.6]{Rib10}, $H^1(J,N) \iso \textstyle \oplus_{p\in \pi(J)} H^1(J,N)_p$, where $\pi(J)$ denotes the set of prime divisors of $\abs{J}$ and $H^1(J,N)_p$ is the $p$-primary component of $H^1(J,N)$ for each prime $p\in \pi(J)$. If additionally $J$ is finite~\cite[Thm.~III.10.3]{Bro82}, we have
\begin{equation}
	\label{eq:ab_decomp}
	H^1(J,N) \iso \textstyle \oplus_{p\in \pi(J)} \inv_J H^1(J_p,N),
\end{equation}
where $J_p$ is a Sylow $p$-subgroup of $J$ and $\inv_J H^1(J_p,N) \iso H^1(J,N)_p$ denotes the set of $J$-invariant elements in $H^1(J_p,N)$ for each prime $p$ (see \S\ref{ss:note} for details). The map $\phi \mapsto \oplus_{p\in \pi(J)} \phi|_{J_p}$ given by the direct sum of the restrictions induces the isomorphism~\eqref{eq:ab_decomp}.

For nonabelian $N$, such a decomposition cannot exist in general, even as an isomorphism of pointed sets. Losey and Stonehewer~\cite[Sec.~3]{Los79} provide the example of $J=S_3$ operating on $N=Q_8$ as in $GL(2,3)$. In this case, there is a complement $J'$ to $N$ that is locally conjugate but not conjugate to $J$; $H^1(J,N)$ has order two while $H^1(J_p,N)$ is trivial for each Sylow $p$-subgroup $J_p$ of $J$. Thus, even for finite groups, requiring $J$ to be solvable or even supersolvable is not sufficient. However, with additional restrictions on $J$ and $N$, we can recover a primary-type decomposition for nonabelian $N$:

\begin{lemma}
	\label{lem:loc_conj}
	For a profinite group $J$ and a finite nilpotent $J$-group $N$, if either $NJ$ is prosupersolvable or $J$ is pronilpotent, the map $\phi \mapsto \times_{p\in\pi(J)} \phi|_{J_p}$ induces an isomorphism $H^1(J,N) \iso \times_{p \in \pi(J)} \inv_J H^1(J_p, N)$ of pointed sets, where  $J_p \in \Syl_p(J)$ for each $p \in \pi(J)$.
\end{lemma}

We also provide an inclusion version of the local conjugacy result for profinite semidirect products:

\begin{corollary}
	\label{thm:loc_inclusion}
	Let $H$ be a closed subgroup of a profinite semidirect product $G=N \semi J$ where $N$ is pronilpotent and either $G$ is prosupersolvable or $J$ is pronilpotent. If $N\cap H \normal N$ and $H$ contains a conjugate of some Sylow $p$-subgroup of $J$ for each prime $p$, then $H$ contains a conjugate of $J$.
\end{corollary}

This in turn allows us to give a fixed-point result for non-coprime actions:
\begin{corollary}
	\label{cor:fix_pt}
	Suppose a profinite semidirect product $G=N \semi J$ acts transitively on some nonempty set $\Omega$ with closed point stabilizers $\{G_\alpha\}_{\alpha\in\Omega}$, where $N$ is pronilpotent and either $G$ is prosupersolvable or $J$ is pronilpotent. If $N_\alpha \normal N$ for some $\alpha\in\Omega$, and for each prime $p$, a Sylow $p$-subgroup of $J$ fixes an element of $\Omega$, then $J$ fixes an element of $\Omega$.
\end{corollary}

Glauberman proved that this result holds whenever $N$ acts transitively and the orders of $N$ and $J$ are coprime, without any further restrictions on $N$ or $J$~\cite[Thm.~4]{Gla64}. Previous studies considered non-coprime actions in finite groups~\cite{Bur25a,Bur25b}. We stress that the hypothesis $N\cap H \normal N$ in Cor.~\ref{thm:loc_inclusion} is necessary. Consider the cyclic group $J=C_6$ acting on the Heisenberg group $N$ of order 27 as in $C_3\wr S_3$. Then $J$ and $N$ are nilpotent, $G\iso N \semi J$ is supersolvable of order 162, and there exists a subgroup $H\iso C_3\times S_3$ of order 18 that contains a conjugate of some Sylow $p$-subgroup of $J$ for each prime $p$ but not a conjugate of $J$.

\subsection{Outline}
We proceed as follows. In the remainder of this section, we review notation. In Section~\ref{s:decomp}, we present the primary-type decomposition given in Lemma~\ref{lem:loc_conj}. We then prove Theorem~\ref{thm:loc_conj} and Corollary~\ref{thm:loc_inclusion} in Section~\ref{s:local_coh}. We conclude in Section~\ref{s:fixed} with a proof of Corollary~\ref{cor:fix_pt} and some final remarks.

\subsection{Preliminaries}
\label{ss:note}

We write $H\leq G$ (resp. $H\normal G$) to denote that $H$ is a closed  (resp. closed normal) subgroup of $G$ and otherwise use standard notation for profinite groups as found in Ribes and Zalesskii~\cite{Rib10} or Wilson~\cite{Wil98}. Given a profinite group $J$ and a discrete $J$-group $N$, we can consider the pointed set $H^1(J,N)$ defined above. As in the case that $N$ is abelian, this set lies in bijective correspondence with $N$-conjugacy classes of compact complements to $N$ in $NJ$.
For $K\leq J$, we let $\res^J_K: H^1(J,N) \to H^1(K,N)$ denote the map induced by restricting elements of $Z^1(J,N)$ to $K$.

For $j\in J$ and $\phi\in Z^1(K,N)$, we define $\phi^j \in Z^1(K^{j^{-1}},N)$ by
\begin{equation}
	\label{eq:cocycle_act_def}
	\phi^j(x) = \phi(x^j)^{j^{-1}}
\end{equation}
for $x\in K^{j^{-1}}$. We say that $\phi$ is $J$-invariant if $\phi|_{K\cap K^{j^{-1}}} \sim \phi^j|_{K\cap K^{j^{-1}}}$ for all $j\in J$ and let $\inv_J H^1(K,N)$ denote the classes of $J$-invariant elements. Note that restricted elements are $J$-invariant, i.e. $\res^J_K H^1(J,N) \subseteq \inv_J H^1(K,N)$, as for any $\phi \in Z^1(J,N)$, we have $\phi^j(x) = n^{-1} \phi(x) n^{x^{-1}}$ for all $x\in J$ where $n=\phi(j)$.

For fixed $\phi \in Z^1(J,N)$, let $N_\phi$ denote the topological group $N$ with twisted $J$-action,
\begin{equation}
	\label{eq:N_twisted}
	n^{j^{-1}}_\phi = \phi(j) n^{j^{-1}} \phi(j)^{-1}
\end{equation}
for $j\in J$. Then, for any $\psi \in Z^1(J,N)$, $\psi\phi^{-1} \in Z^1(J, N_\phi)$. In particular, the map $[\psi] \mapsto [\psi\phi^{-1}]$ yields a bijection between $H^1(J,N)$ and $H^1(J,N_\phi)$ that sends $[\phi ]$ to the distinguished point. For further details on nonabelian group cohomology, see Serre~\cite[Sec.~I.5]{Ser02}.

\section{Primary-type decomposition}
\label{s:decomp}

In this section, we prove Lemma~\ref{lem:loc_conj}. We begin with a useful proposition:

\begin{proposition}
	\label{prop:fixed_cocycle}
	For a profinite group $J$ and a locally finite, discrete $J$-group $N$ that is also a $p$-group, suppose $Q\leq J$ is a procyclic $p_0$-group for some prime $p_0\ne p$ and that $J_0 \normal J$. If $\phi \in Z^1(J_0, N)$ is $Q$-invariant, then there exists $\psi \in Z^1(J_0,N)$ with $\psi \sim \phi$ such that $\psi^q = \psi$ for all $q \in Q$. Furthermore, $\psi(q)=1$ for all $q\in J_0\cap Q$.
\end{proposition}

\begin{proof}
	For $n\in N$, define $\phi_n \in Z^1(J_0,N)$ by $\phi_n(x) = n^{-1}\phi(x)n^{x^{-1}}$ for $x \in J_0$.
	Let $q_0$ generate $Q$ and $n_0\in N$ satisfy $\phi^{q_0} = \phi_{n_0}$. As $J_0$ is compact, $\phi(J_0)$ is finite, and so is
	\begin{equation}
		N_0 = \langle \nu^j: \nu \in \phi(J_0) \cup \{n_0\}, j\in J \rangle,
	\end{equation}
	by the local finiteness of $N$ and the fact that each $J$ orbit is finite as $J$ is compact. Letting $N_1=\{n \in N_0: \phi_n = \phi\}\leq N_0$, the set $S=\{\phi_n\}_{n\in N_0}$ of distinct cocycles cohomologous to $\phi$ via some $n\in N_0$ has cardinality $[N_0:N_1]$, a finite power of $p$. The element $q_0$ permutes the finite set $S$ under the continuous map $q \cdot \phi_n= (\phi_n)^q$ as in~\eqref{eq:cocycle_act_def}. As $S$ is closed in $N^{J_0}$, the fact that $\langle q_0\rangle$ is dense in $Q$ implies $Q$ acts on $S$. Orbits of this action have cardinality 1 or some power of $p_0$. However, $p_0\ne p$, so there must exist some fixed point $\psi$ satisfying $\psi^q = \psi$ for all $q\in Q$. Restricting to $Q$ makes $\psi$ a continuous homomorphism, i.e. $\psi(q_1q_2)=\psi(q_1)\psi(q_2)$ for $q_1,q_2\in J_0 \cap Q$. In particular, the image of $\psi|_{J_0 \cap Q}$ is both a finite $p$-group and a $p_0$-group, and thus trivial.
\end{proof}

With this, we may now show:

\begin{proposition}
	\label{prop:coprime_iso}
	For a prosolvable group $J$ and a locally finite, discrete $J$-group $N$ that is also a $p$-group, suppose $J_0 \normal J$ has prime index $p_0 \ne p$. Then $\res^J_{J_0}: H^1(J, N) \xrightarrow\iso \inv_J H^1(J_0, N)$ is an isomorphism.
\end{proposition}

\begin{proof}
	For $\phi,\psi \in Z^1(J,N)$, suppose $\phi|_{J_0}\sim \psi|_{J_0}$. In the inflation-restriction exact sequence~\cite[\S{}I.5.8]{Ser02} on $N_\phi$ as in~\eqref{eq:N_twisted},
	\begin{equation*}
		1 \to H^1(J/J_0, N_\phi^{J_0}) \to H^1(J, N_\phi) \xrightarrow{\res^J_{J_0}} \inv_J H^1(J_0, N_\phi),
	\end{equation*}
	$H^1(J/J_0, N_\phi^{J_0})$ is trivial so that $\res^J_{J_0}$ maps only the distinguished point to the distinguished point. In particular, $[\psi\phi^{-1}]$ is trivial in $H^1(J, N_\phi)$ so that $\phi \sim \psi $. Consequently, $\res^J_{J_0}$ is injective.

	To show that $\res^J_{J_0}$ is also surjective, suppose $\phi \in Z^1(J_0, N)$ is $J$-invariant.  Write $J=J_0Q$ for some procyclic $p_0$-group $Q$~\cite[\emph{cf.} Lem.~9.8]{Doe92}. By Proposition~\ref{prop:fixed_cocycle}, there exists some $\psi\sim \phi$ such that $\psi^q=\psi$ for all $q\in Q$ and $\psi(q)=1$ for all $q\in J_0\cap Q$. Define $\tilde\psi: J \to N$ as follows. For $j \in J$, write $j=xq$ for some $x\in J_0$ and $q\in Q$, and let $\tilde\psi(j)=\psi(x)$. We claim $\tilde\psi$ is well-defined. If $j=x_0q_0=x_1q_1$ for some $x_0,x_1 \in J_0$ and $q_0,q_1\in Q$, then $x_1^{-1}x_0=q_1q_0^{-1} \in J_0\cap Q$ so that $1 =\psi(x_1^{-1}x_0) = \psi(x_1^{-1})\psi(x_0)^{x_1}$ where $\psi(x_1^{-1})=(\psi(x_1)^{-1})^{x_1}$. Consequently, $\psi(x_0)=\psi(x_1)$ and so $\tilde\psi(x_0q_0) = \tilde\psi(x_1q_1)$, as required. Furthermore, for arbitrary $x_0,x_1 \in J_0$ and $q_0,q_1\in Q$,
	\begin{multline*}
		\tilde\psi(x_0q_0x_1q_1)  = \psi(x_0x_1^{q_0^{-1}}) = \psi(x_0) \psi(x_1^{q_0^{-1}})^{x_0^{-1}} \\
		= \psi(x_0) \psi(x_1)^{(x_0q_0)^{-1}} = \tilde\psi(x_0q_0)\tilde\psi(x_1q_1)^{(x_0q_0)^{-1}},
	\end{multline*}
	where the third equality follows from the fact that $\psi^q=\psi$ for $q\in Q$. Therefore, $\tilde \psi \in Z^1(J, N)$ and $\tilde\psi|_{J_0} = \psi$, so that $\res^J_{J_0}$ is surjective.
\end{proof}

We will also make use of:

\begin{proposition}
	\label{prop:res_pi}
	For a profinite group $J$ and a finite $J$-group $N$ that is also a $p$-group, suppose $NJ$ is prosupersolvable. Let $\pi$ denote the set of primes not exceeding $p$ and let $Q$ be a Hall $\pi$-subgroup of $J$. Then  $\res^J_Q: H^1(J, N) \xrightarrow\iso H^1(Q, N)$ is an isomorphism.
\end{proposition}

\begin{proof}
	Under the hypotheses of the proposition, $J\iso M \semi Q$ where $M \normal J$ is the Hall $\pi'$-subgroup~\cite[Prop.~3.5]{Olt78}. In the inflation-restriction exact sequence,
	\begin{equation*}
		1 \to H^1(J/M, N^M) \xrightarrow{\operatorname{inf}^{J/M}_J} H^1(J, N) \xrightarrow{\res^J_M} \inv_J H^1(M, N),
	\end{equation*}
	$H^1(M, N)$ is trivial so that $H^1(J, N) \iso H^1(Q, N^M)$. As $M$ is also the Hall $\pi'$-subgroup of $NM$, we have $M \normal NM$ so that $NM\iso N \times M$ and in particular $N^M=N$, yielding $H^1(J, N) \iso H^1(Q, N)$. We claim $\res^J_Q$ affords this isomorphism; it suffices to note that $\res^J_Q \circ \operatorname{inf}^Q_J$ is the identity on $H^1(Q, N)$.
\end{proof}

We are now prepared to prove Lemma~\ref{lem:loc_conj}.

\begin{proof}[Proof of Lemma~\ref{lem:loc_conj}]
	Suppose $J$ is profinite and $N$ is a finite nilpotent $J$-group where either $NJ$ is prosupersolvable or $J$ is pronilpotent. The decomposition $N=\times_p N_p$ of $N$ into the direct product of its Sylow $p$-subgroups~\cite[Prop.~2.3.8]{Rib10} induces the isomorphism:
	\begin{equation}
		\label{eq:prod_init}
		H^1(J,N) \iso \times_{p\in \pi(J)} H^1(J,N_p),
	\end{equation}
	as coprime terms drop due to the Schur--Zassenhaus theorem for profinite groups~\cite[Thm.~2.3.15]{Rib10}. For a fixed prime $p\in \pi(J)$, we can then consider the restriction,
	\begin{equation}
		\label{eq:res}
		\res^J_{J_p}: H^1(J,N_p) \to \inv_J H^1(J_p, N_p),
	\end{equation}
	and inclusion,
	\begin{equation}
		\label{eq:inc}
		\inc^N_{N_p}: \inv_J H^1(J_p,N_p) \to \inv_J H^1(J_p, N).
	\end{equation}
	As the composition of the above maps gives us $\phi \mapsto \times_ p \phi|_{J_p}$, it suffices to show that~\eqref{eq:res} and~\eqref{eq:inc} are isomorphisms. For the latter map, we note that $\rho\circ \inc^N_{N_p}$ is the identity on $H^1(J_p,N_p)$ where $\rho: N \to N_p$ is the natural projection, so that $\inc^N_{N_p}$ is injective. We also have the associated exact sequence~\cite[Prop.~38 in \S{}I.5.5]{Ser02}:
	\begin{equation}
		\label{eq:inc_exact}
		H^1(J_p,N_p) \xrightarrow{\inc^N_{N_p}} H^1(J_p, N) \to H^1(J_p, N/N_p),
	\end{equation}
	where $H^1(J_p, N/N_p)$ is trivial so that $\inc^N_{N_p}$ in~\eqref{eq:inc_exact} is surjective and thus an isomorphism on $H^1(J_p,N_p)$. As $\inc^N_{N_p}$ and the projection $\rho$ preserve $J$-invariance, we may conclude that~\eqref{eq:inc} is an isomorphism. We now focus on the restriction~\eqref{eq:res}.

	\emph{For injectivity}, suppose $\phi\nsim\psi$ in $Z^1(J,N_p)$ but $\phi|_{J_p} \sim \psi|_{J_p}$. Then the collection $\mathcal L = \{ L \leq J: J_p \leq L, \phi|_L \nsim \psi|_L\}$ contains $J$ but not $J_p$. For any nonempty chain $(L_i)_{i \in I}$ in $\mathcal L$, $L'=\cap_{i\in I} L_i$ is closed and satisfies $J_p \leq L' \leq J$. Suppose by way of contradiction that for some $n\in N_p$, we had $\psi(x) = n^{-1} \phi(x) n^{x^{-1}}$ for all $x\in L'$. Then $E = \{ x \in J:  \psi(x) = n^{-1} \phi(x) n^{x^{-1}} \}$ would be an open set containing $L'$. The sets $\{L_i \setminus E\}_{i\in I}$ would be closed with empty intersection, implying $L_j \setminus E$ is empty for some $j\in I$ by the compactness of $J$. It would follow that $\phi|_{L_j} \sim \psi|_{L_j}$, contradicting $L_j \in \mathcal L$. In particular, Zorn's lemma implies that there exists some $L$ minimal with respect to inclusion such that $\phi\nsim\psi$ in $Z^1(L,N_p)$ but $\phi|_{J_p} \sim \psi|_{J_p}$. Replacing $J$ with $L$ if necessary, we may assume $J$ is minimal.

	If $N_pJ$ is prosupersolvable and $p$ is not the largest prime divisor of $\abs{J}$, then Prop.~\ref{prop:res_pi} implies that $\res^J_Q: H^1(J, N_p) \xrightarrow\iso H^1(Q, N_p)$ for some Hall $\pi$-subgroup $Q$ of $J$ containing $J_p$ where $\pi$ is the set of primes not exceeding $p$. By the minimality of $J$, $\phi|_Q\sim\psi|_Q$, so that $\phi \sim\psi$, a contradiction. Otherwise, let $J_0 \normal J$ be of prime index $p_0\ne p$. If $N_pJ$ is prosupersolvable, $p_0$ may be taken to be the smallest prime divisor of $J$. Again by the minimality of $J$, we must have $\phi|_{J_0}\sim\psi|_{J_0}$. Proposition~\ref{prop:coprime_iso} then implies $\res^J_{J_0}: H^1(J, N_p) \xrightarrow\iso \inv_J H^1(J_0, N_p)$ is an isomorphism, so that $\phi\sim\psi$, a contradiction. We conclude that~\eqref{eq:res} is injective.

	\emph{For surjectivity}, suppose $[\tau]\in  \inv_J H^1(J_p, N_p) \setminus \res^J_{J_p} H^1(J,N_p)$. The collection $\mathcal L = \{L \leq J: J_p \leq L, [\tau] \notin \res^{L}_{J_p} H^1(L,N_p)\}$ contains $J$ but not $J_p$. For any nonempty chain $(L_i)_{i \in I}$ in $\mathcal L$, $L'=\cap_{i\in I} L_i$ is closed and satisfies $J_p \leq L' \leq J$. Suppose by way of contradiction that $[\tau] \in \res^{L'}_{J_p} H^1(L',N_p)$ so that $\tilde\tau|_{J_p}=\tau$ for some $\tilde\tau \in Z^1(L', N_p)$. As $L'$ is compact and $N_p$ is discrete, $\tilde\tau(L')$ is finite. As $\tilde\tau^{-1}(1)$ is open in $L'$, there exists some open $E\normal J$ such that $E \cap L' \leq \tilde\tau^{-1}(1)$ and $n^e=n$ for all $n \in \tilde\tau(L')$, $e\in E$. Then $EL' \leq J$ and we can define $\hat \tau(ej)=\tilde \tau(j)$ for $e \in E, j\in L'$. It may be verified that $\hat \tau \in Z^1(EL', N_p)$ and $\hat \tau |_{L'} = \tilde \tau$. As in the previous argument, it follows that $EL'$ is an open subset containing $L'$, so compactness implies $L_j \leq EL'$ for some $j\in I$, contradicting $L_j \in \mathcal L$. Zorn's lemma again gives us some $L$ minimal with respect to inclusion such that $[\tau]\in  \inv_J H^1(J_p, N_p) \setminus \res^{L}_{J_p} H^1(L,N_p)$. We again proceed under the assumption that $J$ is minimal.

	As before, if $N_pJ$ is prosupersolvable and $p$ is not the largest prime divisor of $\abs{J}$, $\res^J_Q: H^1(J, N_p) \xrightarrow\iso H^1(Q, N_p)$ for some Hall $\pi$-subgroup $Q$ of $J$ containing $J_p$. By the minimality of $J$, $[\tau] \in \res^{Q}_{J_p} H^1(Q,N_p)$, implying that $[\tau] \in \res^{J}_{J_p} H^1(J,N_p)$, a contradiction. In the other case, we may then find $J_0 \normal J$ of prime index $p_0 \ne p$. Again by the minimality of $J$, there exists $\tilde \tau \in Z^1(J_0,N_p)$ such that $\tilde\tau|_{J_p} = \tau$. Write $J=J_0Q$ for some procyclic $p_0$-group $Q$. As $J_p\normal J$, we have, for $q\in Q$,
	\[
		\tilde\tau^q|_{J_p}
		= \left(\tilde\tau|_{J_p}\right)^q
		= \tau^q
		\sim \tau
		= \tilde\tau|_{J_p},
	\]
	so that the injectivity of $\res^{J_0}_{J_p}$ implies $\tilde \tau$ is $Q$-invariant, and consequently $J$-invariant. Prop.~\ref{prop:coprime_iso} again implies that $[\tau] \in \res^{J}_{J_p} H^1(J,N_p)$, yielding our final contradiction and completing the proof.
\end{proof}

\section{Applications to local conjugacy}
\label{s:local_coh}

In this section, we use the primary-type decomposition from the previous section to prove Theorem~\ref{thm:loc_conj} and Corollary~\ref{thm:loc_inclusion}. In the case of complements, the conjugacy result for finite $N$ is nearly immediate:

\begin{proposition}
	\label{prop:comps}
	Given $N$ and $G$ satisfying the hypotheses of Theorem~\ref{thm:loc_conj} where $N$ is finite, two closed complements of $N$ are conjugate if and only if they are locally conjugate.
\end{proposition}

\begin{proof}
	Suppose $J$ and $J'$ are locally conjugate complements of $N$. Let $\phi \in Z^1(J,N)$ correspond to $J'$. For each prime $p$, we have $\phi|_{J_p} \sim 1|_{J_p}$ for some $J_p \in \Syl_p(J)$ where $1\in Z^1(J,N)$ denotes the distinguished point. Lemma~\ref{lem:loc_conj} implies that $\phi\sim 1$ so that $J$ and $J'$ are conjugate.
\end{proof}

We begin the proof of Theorem~\ref{thm:loc_conj} with the case that $N$ is finite.

\begin{proposition}
	\label{prop:thm_main_n_nilpotent}
	If $N$ is finite, Theorem~\ref{thm:loc_conj} holds.
\end{proposition}

\begin{proof}
	We induct on $\abs{N}$. In the next two paragraphs, we adapt Losey and Stonehewer's original arguments to justify two simplifying assumptions.

	\emph{We may assume $N$ is a $p$-group}. If multiple primes divide $\abs{N}$, then for some prime $p$ we have the nontrivial decomposition $N=N_p \times N_{p'}$ for $N_p\in \Syl_p(N)$ and $N_{p'}$ a Hall $p'$-subgroup of $N$. In $G/N_p$, we have $H^{n_1} \leq N_p H'$ for some $n_1 \in N_{p'}$. Similarly, in $G/N_{p'}$, it follows that $H^{n_0} \leq N_{p'}H'$ for some $n_0 \in N_p$ so that $H^{n_0n_1} \leq N_pH' \cap N_{p'}H'$. If $g \in  N_pH' \cap N_{p'}H'$, then $g = n_2h_2 = n_3h_3$ for some $n_2 \in N_p$, $n_3 \in N_{p'}$, and $h_2, h_3 \in H'$. Then, $h_3h_2^{-1}=n_3^{-1}n_2 \in H'$ where $n_2$ and $n_3$ commute and have coprime orders so that $n_2, n_3 \in H'$. It follows that $N_pH' \cap N_{p'}H' = H'$ so that $H^{n_0n_1} \leq H'$. A second application of this argument implies that $H' \leq H^{n_4}$ for some $n_4\in N$. As $[G:H] = [N: N\cap H] < \infty$, and similarly $[G:H']<\infty$, it follows that $H^{n_0n_1} = H'$.

	\emph{We may assume $H$ and $H'$ generate $G$}. As $N$ is a $p$-group, we may suppose that $H$ and $H'$ share a common Sylow $p$-subgroup, say $P$. If $G_0 = \langle H, H' \rangle$ is strictly contained in $G$, then $H$ and $H'$ remain locally conjugate in $G_0$, as for all other primes $q\ne p$, Sylow $q$-subgroups of $H$ and $H'$ are also Sylow $q$-subgroups of $G_0$. However, $N\cap G_0 \lneq N$ so that the inductive assumption implies $H$ and $H'$ are conjugate.

	We now proceed under the above assumptions with $P \in \Syl_p(H) \cap \Syl_p(H')$. As $N\cap H \leq P$, $N \cap P = N \cap H \normal H$. Similarly, $N \cap P = N \cap H' \normal H'$. As $G=\langle H, H' \rangle$, it follows that $N\cap H = N \cap H'$ is normal in $G$. If $N\cap H$ is nontrivial, induction in $G/(N\cap H)$ allows us to conclude. Otherwise, $N\cap H$ is trivial, so that $H$ and $H'$ complement $N$ in $G$ and we may conclude by Proposition~\ref{prop:comps}.
\end{proof}

We are now prepared to prove Theorem~\ref{thm:loc_conj}.

\begin{proof}[Proof of Theorem~\ref{thm:loc_conj}]

	Given $N$ and $G$ as described, suppose $H$ and $H'$ are closed, locally conjugate supplements of $N$ in $G$. Let $\mathcal U$ denote the set of open normal subgroups of $G$. For each $U\in\mathcal U$, $N_U=N/(N\cap U)$ is nilpotent so that Proposition~\ref{prop:thm_main_n_nilpotent} implies the images of $H$ and $H'$ are conjugate in $G/(N\cap U)$. Thus, there exists $x_U \in G$ such that $(N\cap U)H^{x_U}=(N\cap U)H'$. In particular, for each $U\in \mathcal U$,
	\begin{equation*}
		K_U = \{x \in G: (N\cap U)H^x=(N\cap U)H' \} = N_G\big((N\cap U)H\big) x_U
	\end{equation*}
	is closed~\cite[Exer.~0.4(2)(a)]{Wil98} and nonempty. Furthermore, $K_U \subseteq K_V$ if $U\subseteq V$. It follows that $\{K_U\}_{U \in \mathcal U}$ has the finite intersection property. As $G$ is compact, $\cap_{U\in\mathcal U} K_U$ is non-empty and so contains some element, say $y$. Then, as $\mathcal U$ is filtered from below~\cite[Prop.~2.1.4]{Rib10},
	\[
		H^y = \big(\cap_{U\in\mathcal U} (N\cap U)H\big)^y
		= \cap_{U\in\mathcal U} (N\cap U)H^y
		=  \cap_{U\in\mathcal U} (N\cap U)H'
		= H'
	\]
	so that $H$ and $H'$ are conjugate.
\end{proof}

We now prove Corollary~\ref{thm:loc_inclusion}.

\begin{proof}[Proof of Corollary~\ref{thm:loc_inclusion}]
	Given $H$ and $G=N\semi J$ as described in the corollary, we note that $H$ supplements $N$ in $G$, as $NH$ contains a Sylow $p$-subgroup of $G$ for each prime $p$.
	As $N\cap H \normal H$, it follows that $N\cap H \normal G$. In $G/(N\cap H)$, $H/(N\cap H)$ and $J(N\cap H)/(N\cap H)$ are locally conjugate complements, so that Theorem~\ref{thm:loc_conj} implies they are conjugate. It follows that $J(N\cap H)$ and $H$ are conjugate in $G$, allowing us to conclude.
\end{proof}

\section{Fixed-point results and conclusions}
\label{s:fixed}

We now prove Corollary~\ref{cor:fix_pt}.

\begin{proof}[Proof of Corollary~\ref{cor:fix_pt}]
	Given $G=N \semi J$ and $\Omega$ as described in the hypotheses of the corollary, let $G_\alpha$ denote the stabilizer subgroup of $G$ for some $\alpha\in\Omega$. For each prime $p$, it follows that $(J_p)^{n_p} \leq G_\alpha$ for some $J_p \in\Syl_p(J)$ and $n_p\in N$. Corollary~\ref{thm:loc_inclusion} allows us to conclude $J^g \leq G_\alpha$ for some $g\in G$. In particular, $J$ fixes $g\cdot\alpha$.
\end{proof}

Notably, any local inclusion result along the lines of Corollary~\ref{thm:loc_inclusion} allows us to obtain an analogous fixed-point result. For abelian $N$, we have the following.

\begin{proposition}
	\label{thm:loc_inclusion_ab}
	In a profinite group $G$, suppose $H, H' \leq G$ each supplement some abelian $N \normal G$. If for each prime $p$, $H$ contains a conjugate of some Sylow $p$-subgroup of $H'$, then $H$ contains a conjugate of $H'$.
\end{proposition}

\begin{proof}
	First suppose $N$ is finite and induct on $\abs{N}$. As in the proof of Proposition~\ref{prop:thm_main_n_nilpotent}, we may assume that $N$ is a $p$-group. If $N\cap H \normal G$ is nontrivial, we may apply induction in $G/(N\cap H)$ to conclude. Otherwise, $N \cap H$ is trivial, so that $N \cap H'$ is as well. It follows that $H$ and $H'$ are locally conjugate complements of $N$ and thus conjugate by Shin's result~\cite{Shi95}. The argument for non-finite $N$ mirrors the proof of Theorem~\ref{thm:loc_conj}.
\end{proof}

This in turn gives us an analogue of Corollary~\ref{cor:fix_pt} for abelian subgroups $N$:
\begin{proposition}
	\label{thm:fix_pt_ab}
	Suppose a profinite group $G$ acts transitively and with closed point stabilizers on some nonempty set $\Omega$ and that $H\leq G$ supplements some abelian $N\normal G$. If for each prime $p$, a Sylow $p$-subgroup of $H$ fixes an element of $\Omega$, then $H$ fixes an element of $\Omega$.
\end{proposition}

\subsection*{Acknowledgements}
The author thanks Gareth Tracey for his invitation to speak at the University of Warwick's Algebra Seminar and the warm hospitality he received there. He is also grateful to Mark Grant for his correspondence on the necessity of conditions in Lemma~\ref{lem:loc_conj} and to Pierre-Emmanuel Caprace for feedback on the manuscript.


\end{document}